\documentclass[envcountsame,envcountchap]{article}
\usepackage{amssymb,amsmath}
\usepackage{amsfonts}
\usepackage{graphicx}
\usepackage[matrix,arrow,ps]{xy}

\usepackage{graphicx}
\usepackage{subfigure}
\usepackage[usenames]{color}
\usepackage{amssymb,amsmath,amsthm}
\usepackage{amsfonts}
\usepackage{graphicx}
\usepackage{type1cm}
\usepackage{eso-pic}
\usepackage{color}
\usepackage{makeidx}
\usepackage[section]{placeins}
\usepackage{float}

\newtheorem{theorem}{Theorem}[section]

\newcommand{\askip}{\htab\htab {\rm and} \htab\htab}

\newcommand{\emptysett}{\mbox{\O}}

\begin{document}

\newcommand{\Ito}{It\^{o} }
\newcommand{\bz}{{\bf z}}
\newcommand{\vv}{{\bf v}}
\newcommand{\ww}{{\bf w}}
\newcommand{\yy}{{\bf y}}
\newcommand{\xx}{{\bf x}}
\newcommand{\nn}{{\bf n}}
\newcommand{\uu}{{\bf u}}
\newcommand{\mm}{{\bf m}}
\newcommand{\qq}{{\bf q}}
\newcommand{\aba}{{\bf a}}
\newcommand{\pp}{{\bf p}}
\newcommand{\OO}{\mathbb{O}}
\newcommand{\II}{\mathbb{I}}
\newcommand{\IR}{\mathbb{R}}
\newcommand{\IC}{\mathbb{C}}
\newcommand{\IB}{\mathbb{B}}
\newcommand{\IZ}{\mathbb{Z}}
\newcommand{\half}{\frac{1}{2}}
\newcommand{\halff}{1/2}
\newcommand{\bea}{\begin{eqnarray*}}
\newcommand{\eea}{\end{eqnarray*}}
\newcommand{\beaq}{\begin{eqnarray}}
\newcommand{\eeaq}{\end{eqnarray}}
\newcommand{\bfalpha}{\mbox{\boldmath $\alpha$ \unboldmath} \hskip -0.05 true in}
\newcommand{\bfgamma}{\mbox{\boldmath $\gamma$ \unboldmath} \hskip -0.05 true in}
\newcommand{\bfmu}{\mbox{\boldmath $\mu$ \unboldmath} \hskip -0.05 true in}
\newcommand{\bfnu}{\mbox{\boldmath $\nu$ \unboldmath} \hskip -0.05 true in}
\newcommand{\bfxi}{\mbox{\boldmath $\xi$ \unboldmath} \hskip -0.05 true in}
\newcommand{\bfphi}{\mbox{\boldmath $\phi$ \unboldmath} \hskip -0.05 true in}
\newcommand{\bfpi}{\mbox{\boldmath $\pi$ \unboldmath} \hskip -0.05 true in}
\newcommand{\beq}{\begin{equation}}
\newcommand{\eeq}{\end{equation}}
\newcommand{\bfomega}{\mbox{\boldmath $\omega$ \unboldmath} \hskip -0.05 true in}
\newcommand{\bfeta}{\mbox{\boldmath $\eta$ \unboldmath} \hskip -0.05 true in}
\newcommand{\bfepsilon}{\mbox{\boldmath $\epsilon$ \unboldmath} \hskip -0.05 true in}
\newcommand{\bftheta}{\mbox{\boldmath $\theta$ \unboldmath} \hskip -0.05 true in}
\newcommand{\om}{\omega}
\newcommand{\Om}{\Omega}
\newcommand{\bom}{\bfomega}
\newcommand{\III}{\rm III}
\def\tab{ {\hskip 0.15 true in} }
\def\vtab{ {\vskip 0.1 true in} }
 \def\htab{ {\hskip 0.1 true in} }
 \def\ntab{ {\hskip -0.1 true in} }
 \def\vtabb{ {\vskip 0.0 true in} }

\begin{center}
{\bf \LARGE Entropy, Symmetry, and the Difficulty of Self-Replication}
\end{center}

\begin{center}
Gregory S. Chirikjian \\
National University of Singapore \\
7 February 2022
\end{center}

\begin{abstract}
The defining property of an artificial physical self-replicating system, such as a self-replicating robot,
is that it has the ability to make copies of itself from basic parts. Three questions that immediately arises in
the study of such systems are:
1) How complex is the whole robot in comparison to each basic part ?
2) How disordered can the parts be while having the robot successfully replicate ?
3) What design principles can enable complex self-replicating systems to function in disordered environments generation after generation ?
Consequently, much of this article focuses on exploring different concepts of entropy as a measure of disorder,
and how symmetries can help in reliable self replication, both at the level of assembly (by reducing the number of wrong ways that parts could be assembled), and also as a parity check when replicas manufacture parts generation after generation. The mathematics underpinning these principles that quantify artificial physical self-replicating systems are articulated here by integrating ideas from
information theory, statistical mechanics, ergodic theory, group theory, and integral geometry.
\end{abstract}

\section{Introduction}

Artificial self-replicating systems have the potential to greatly enhance mankind's ability to
harvest resources in outer space. The basic idea is that a functional factory consisting of robots,
materials processing units, and manufacturing apparatuses are formulated to harvest materials and to build
infrastructure from resources found in situ before sending humans into unstructured extraterrestrial environments. This concept is by no means new, and the history can be found
in many prior technical works including \cite{f1,f2,f3,c1,l1,l2,moses2020}.
The particular presentation in this paper is an embellishment of the author's older works \cite{pk_case,pk_dss}.

The basic difficulty is that in order to close the self-replicating loop, there must be a balance between
the overall system complexity and the ability of the system to handle the simplest inputs possible. Generally
speaking, the more complex the system is, the more capable it is to assemble the simplest parts, or even harvest
raw materials. But when the system is very complex, it then requires more to reproduce. For example, if a robot needs microprocessors and a 3D printing system needs a high-energy laser, then producing those from in situ resources becomes an enormous challenge. One way around that problem is to focus on production of mechanical parts in situ such as structural members for robots, factories, and habitats, and to send the relatively light-weight ``vitamin'' components such as computers, lasers, and some chemical reagents to be considered as
inputs to the system rather than as items to be replicated. In this way, a practical version of
in situ resource utilization (ISRU) can be achieved by reclassifying inputs. Moreover, developing systems that are dependent on some inputs that the systems cannot manufacture from scratch is both more akin to living systems that require specific nutrients, and is also safer than developing self-replicating systems that could continue to replicate without the imposition of resource bottlnecks. Without such intentional bottlenecks there could be fears of things running amuck for generations.

The emphasis of this article is not to produce or review any particular system architecture. This has already
been done elsewhere recently \cite{moses2020}. Instead, this article seeks to enumerate mathematical modeling tools that may be useful in assessing any self-replicating system. These tools can be applied to artificial and biological systems. In particular, concepts of entropy and complexity that enter in many fields can be used here. Entropy as a measure of disorder appears in two different forms in classical thermodynamics - as discrete sums over microstates and as integrals that approximate. In classical and quantum probability and in information theory, entropy plays a prominent role. The various concepts of entropy are relevant to artificial self-replicating systems because of the interplay between the physical and informational. (It can be argued that
information is physical, but there are some aspects of information that are independent of its particular physical implementation.)

The {\it degree of self replication} is a concept to measure how complex a self-replicating system is \cite{l2}.
Degree of self replication is additive for modular parts but can be super-additive if synergy results from adding passive components to make machines. For example, assembling gears to form a transmission gives the resulting assemblage a functionality that transcends counting the number of parts that constitute it. There are many ways to define system complexity that can be used, such as algorithmic complexity theory, and these
concepts are reviewed as well. (In the example of assembling a transmission, the number of steps required to do the assembly might be counted instead of the number of input parts.)

The theory of artificial self-replicating automata has a long history, and elements were touched on by
Turing, Von Neumann, Shannon, Liang, Arbib, and Freitas, as reviewed in \cite{moses2020}. Self replication
is also related to the theory of complexity as articulated in \cite{Zurek, Pincus, Teix, Zur2,Zurek2}, the thermodynamics
of computation \cite{Benn, Sz}, and physical aspects of information theory \cite{Br, Land1, Land2}.
Many of the landmark papers in this area have been collected in the books \cite{st_leff1,st_leff2,Zurek1}.

Most work in the area of artificial self-replicating systems is in the form of artificial life in silico. The challenges in a-life are somewhat different than those involved in the design of systems that must replicate in a material sense in the 3D physical world governed by mechanics and thermodynamics. This is the emphasis of the current work. Other theoretical aspects related to the foundations of self-replicating machines can be found in
\cite{menezes,sayama}.

The remainder of this paper is structured as follows. Section \ref{degreesec} reviews the concept of
`degree of self replication.' Section \ref{entropysec} explains why the concept of entropy is appropriate
to describe the robustness of a self-replicating system in the presence of disorder, and quantitatively compares
different concepts of entropy from different fields. Section \ref{symmetry} explains how symmetry, and the mathematical discipline of group theory, can be used to describe how design considerations can make the process of self replication easier. Section \ref{kinematic} explains how a formula from the field of integral geomety,
the principal kinematic formula, can be used to describe how constraining the motion of parts reduces entropy to make assembly easier. Section \ref{error} discusses how errors propagate over generations of self replication and
how symmetry can be used as an error correction mechanism.

\section{Degree of Self Replication} \label{degreesec}

A question that arises naturally in the context of artificial self-replicating systems is ``In comparison to a biological systems, is the artificial system really self-replicating ?''. In other words, a biological system such as a bacterium takes in nutrients and produces a copy of itself. The issue at hand is that the nutrients appear to be invisible to the naked eye and then the replica pops into existence, as if by magic. However, in fact the nutrients are themselves complex molecules such as amino acids, nucleic acids, lipids, sugars, and even trace metals such as iron, sodium, and potassium. And of course this all happens in the presence of abundant water, which is another resource. Without these essential ingredients a biological system would not replicate, since biological replication is essentially a biomolecular process.

In an artificial self-replicating system, if the number of input parts is small, and if each part is itself a complex machine, then the effort involved in the assembly step of self-replication is relatively low compared to the effort of making the input parts. The assembly task in self-replication is then relatively unchallenging in such a scenario. In contrast, if there are many simple (easy to manufacture) parts that need to be assembled by the robot, the relative difficulty is higher. The difference between these two scenarios can be quantified using the concept of the `degree of self replication'. The highest degree of self replication is when raw materials are harvested by a system and used to produce a replica, since fabrication and assembly are both done by the self-replicating system in that scenario.

That is, an artificial system that picks and places a few complex macroscopic modular units and assembles them into a functional copy seems to be far less impressive than its biological counterparts. But it has been argued previously by the author and coworkers that this is a difference in degree rather than type \cite{l2}. That is, a biological replicator has a high degree of self replication whereas current robotic self-replicating systems have a low degree. This can be quantified quite simply. If a system is modular, each module can be assigned a complexity value by counting its constituent parts (or the number of steps required to assemble the parts), and the ratio of complexity of the overall system to the complexity of each  constituent part can be computed. For example, the complexity of a bacterium may consist of millions of individual
amino acids, lipids, and nucleotide bases. If these are considered as the fundamental building blocks, then the degree of self replication might be millions.

In contrast, an artificial self replicating system consisting of modules composed of motors, microprocessors, and sensors, each of which was built in a large factory, and meticulously assembled by a skilled worker before the
robot does a simple pick and place operation seems to pale in comparison. This is reflected in the degree of self replication
\begin{equation}
DOSR \,=\, \frac{System \,\, Complexity}{Part \,\, Complexity}
\label{dosr}
\end{equation}
For the biological system, this might be millions. For the artificial physical replicators to date
this might be tens or hundreds. Therefore there are several orders of magnitude difference reflected
in this simple concept. Eventually, if an artificial replicator could scoop up sand and other raw materials,
and cast parts in molds or 3D print, the degree of self replication for artificial systems could rival those
of biological replicators.

Of course there are subtleties regarding how to measure complexities in the simple formula above. For example, in an amino acid one could count atomic nuclei. But for glycine, there are far fewer atoms than in the amino acids lycene or tryptophan. Then there issues related to metabolism and synergy. Assembly of amino acids to form a polypeptide in a ribosome requires metabolic energy, and perhaps that is a better measure of the effort involved than the geometric arrangement of atomic nuclei since nature produces amino acids from nonbiological processes for use in biological systems. In contrast, in an artificial
self-replicating system, any modular component containing a microprocessor represents a very high concentration
of complexity (as measured either in terms of features, or manufacturing effort, or effort to build the microprocessor factory) prior to the assembly process, thereby limiting the degree of self replication.

When there is a wide range of input part complexities, it makes sense to take the ratio of the overall system complexity to the {\it most complicated} input part rather than an average complexity. Again, the issue of how to measure complexity is a whole other issue. Several options are: 1) to count the number of solid nonarticulated parts that constitute an input object or subsystem, perhaps weighted by the geometric complexity of the part; 2) to count the number of operations or total effort/energy
required to produce an input object; 3) to use algorithmic complexity theory to count the length of the computer program used to produce the input parts (if they are produced by an automated fabrication system).
When measuring the complexity of an assemblage of parts, there are also several ways. For example, one could
simply sum up the complexities of individual parts. But this may underestimate the overall system complexity resulting from synergetic interactions of parts and the effort involved in achieving this synergy. For example,
assembling gears to form a transmission might require precise operations requiring substantial effort, beyond counting the number of features of the part, such as gear teeth. And the overall functionality of the resulting assemblage may not be reflected by simply summing the complexity of the constituent parts.
Therefore it might be more appropriate to consider the algorithmic complexity of the computer programs used to
create parts and to assemble them. Or, one could count the total energy required to produce input parts from the
stage of raw materials, and to then count the energy required to assemble the parts.
Regardless of how complexity is defined, (\ref{dosr}) can be used to measure the degree of self replication in the sense defined by each concept of complexity.

\section{Entropy} \label{entropysec}

Another dimension to the difficulty of self replication is how ordered the input parts are. If they are arranged precisely in an array and the self-replicating system need only pick and place, this is less challenging than if the robot must identify, manipulate, and assemble randomly oriented parts. Moreover, if the parts have symmetries, this reduces the number of wrong ways that they can be assembled. This aspect of the difficulty of self replication can be quantified by computing entropy. Many different concepts of entropy exist from statistical mechanics, information theory, and quantum probability. Since a self replicating system
intertwines physicality and information processing, concepts from these different fields become relevant.
For example, an unordered array of manufactured parts in a bowl can be described using the configurational
entropy of parts. The entropy of a mixture of raw materials that must be separated, refined, and then used to produce finished parts is more related to thermodynamic and ergodic-theoretic entropy.

During a self-replication process, parts are assembled to create the replica. If $N$ parts can exist in many different initial positions and orientations, then the set of all possible initial arrangements can be described as a probability density. When the parts are assembled properly, the disorder will decrease. Entropy is a measure of how disordered the initial and final states are.

Many different concepts of entropy exist. Two broad categories of entropy concepts are
continuous and discrete. The distinction between the two can be blurred by realizing that when
the problems at hand have a natural length scale, features smaller than a particular limit are not measurable.
Therefore whether one creates bins to discretize at this scale, or calls the problem continuous does not matter
because entropy differences are essentially the same whether computed using continuous or discrete formulae.

Aside from the continuous/discrete characterization, there are distinct concepts of entropy in several different fields including information theory, statistical mechanics, ergodic/measure theory, and quantum probability. These various concepts are reviewed here with an eye towards their relevance to artificial self-replicating systems.

The various concepts of entropy examined herein share a number of common features:
Entropy is reduced by environmental obstacles and containment, as well as fences, and by considering symmetry.
In all cases conditioning reduces entropy. Continuous and discrete entropy differences can be equated if bin size is below feature resolution.

Both concepts (entropy and degree of self replication) are influenced by symmetry.
Symmetry reduces the entropy difference between an initial disordered ensemble and the final assemblage, making assembly tasks easier. Coset and double coset space decompositions and associated concepts of fundamental domains quantify this idea. When there order of operations is not critical this also makes assembly tasks easier.

The remainder of this section is structured as follows. In Subsection \ref{discretequantum} the basic defintions
and concepts of discrete and quantum entropy are reviewed. In Subsection \ref{continuousentropy} continuous entropy on a measurable space is defined, and its relationship to discrete entropy is explained. In Subsection
\ref{entdifsec} the properties of entropy differences and relative entropy are reviewed.

\subsection{Discrete and Quantum Entropy} \label{discretequantum}

An {\it event} is a discrete occurrence of some kind.
For example, it can be the appearance of a printed symbol on a page, a sequence of dots and dashes in a telegraph message, a string of binary numbers stored in a computer memory, the process of pouring a discrete quantum of molten material into a mold to form a mechanical part, or the fastening of two mechanical parts in a mechanized assembly procedure. The concept of an event is extremely general.

Let $E_i$ denote such an event and consider a finite collection of events $C = \{E_1, E_2, ..., E_n\}$.
Associated with each event assign a probability $p_i \doteq p(E_i) \geq 0$ and $\sum_{i=1}^{n} p_i = 1$.
Then the {\it self-information} of each such event is
\beq
{\cal I}(E_i) = - \log p(E_i).
\eeq
The base of the logarithm is largely irrelevant for our purposes.
The choice of base is essentially a choice of units, similar to how length can be measured in
feet or meters.

The concept of ${\cal I}(E_i)$ was developed when information fed into ticker tapes or punch cards were described by binary punches in paper, wherein ${\cal I}(E_i)$ describes the amount of tape or cards to encode a message, and this concept carried over to the digital computer. In such concepts it is therefore natural to use a logarithm of base 2, and to measure information in `bits'. In other contexts it is more natural to use base e, and measure information in `nats'.

The primary purpose of the development of information theory was to maximize the amount of information that could reliably be transmitted from one time and place to another. This involves the efficient packaging of information with codes that perform self checks to detect and correct errors. Information and its transmission has many manifestations, including the genetic code and DNA replication and transcription. In the context of self-replicating robotic systems, the spatial arrangment of parts and assembly operations required to produce a replica can be considered as messages that are transmitted. The design of physical means for error correction
while efficiently replicating are then analogous to codes that protect data transmission at the highest rates possible.

The {\it entropy} of the discrete collection $C$ is defined to be the average self-information of all events in $C$:
\beq
H(C) \doteq \sum_{i=1}^{n} p(E_i) \, {\cal I}(E_i) = - \sum_{i=1}^{n} p(E_i) \log p(E_i).
\eeq
Sometimes this is denoted as a functional of the probabilities instead of the underlying events. That is,
if $p =\{p_1,p_2,...,p_n\}$ is the collection of probabilities, then
\beq
S(p) \doteq -\sum_{i=1}^{n} p_i \, \log p_i \,.
\label{spdef}
\eeq
Obviously, $S(p) = H(C)$. The difference in notation only reflects a difference in emphasis regarding the meaning of entropy.

The above concepts have analogs in quantum probability. In the simplest case, if a diagonal matrix $P$ is defined with the values $\{p_i\}$ on its diagonal, then the (von Neumann) quantum entropy
\beq
S(P) \,\doteq\, - {\rm tr}\left(P \log P\right)
\eeq
Where $\log P$ is the matrix logarithm and the terms inside of the trace is the product of two matrices.
Whereas the above gives exactly the same result as the discrete entropy $S(p)$, it is more general because
$P$ can be taken to be any symmetric real matrix (or even a complex Hermitian one) with the condition
that ${\rm tr}(P)=1$. Hermitian matrices of this sort arise in quantum mechanics, which is the domain in which
von Neumann defined quantum entropy. So far, a connection between self-replicating systems and quantum entropy
has not been explored, and the concept is included here only for completeness.

\subsection{Entropy of Probability Density Functions on Continuous Measure Spaces} \label{continuousentropy}

A probability density function (pdf) on a measure space $(X,\mu)$
is a function $f:X \,\longrightarrow\, \IR_{\geq 0}$ such that
$$ \int_X f(\xx) \, d\mu(\xx) \,=\,1\,. $$
The entropy of $f$ is defined as
\beq
S(f) \,\doteq\, -\int_X f(\xx) \log f(\xx) \, d\mu(\xx) \,.
\label{contentr}
\eeq
This is the continuous analogue of (\ref{spdef}). Again,
the choice of base of the logarithm amounts to a choice of measurement units for $S$.
Throughout this section, $\log \doteq \log_e = \ln$.

In the case when $\qq$ parameterizes the group of rotations or the Euclidean motion group,
then the entropy of a pdf $f:G \rightarrow\,\IR_{\geq 0}$ is
\beq
S_G(f) \,\doteq\, -\int_{G} f(g) \log f(g) \,dg
\label{Gent}
\eeq
where the natural bi-invariant measure $dg$ takes on different appearances under changes
of parametrization but the value of the integral is independent of parametrization and it
is invariant under shifts of the form $g \rightarrow g_0 g$ and $g \rightarrow g g_0$.

More than 35 years ago, A. Sanderson quantified the concept of part disorder by defining the concept of ``parts entropy'' \cite{pk_partent}, which motivates the definition in (\ref{Gent}).

\subsubsection{Statistical Mechanics, Configurational Entropy, and Discretization} \label{statmech}

In statistical mechanics, the continuous measure space is $2n$-dimensional where $n$ is the number
of degrees of freedom of the physical system under consideration. This ``phase space''
has a volume element of the form $d\mu(\xx) = d{\bf p} d{\bf q}$ where ${\bf q}$ denotes a set of generalized
coordinates and ${\bf p}$ denotes the corresponding conjugate momenta. This form is invariant under changese
of coordinates. In statistical mechanics, entropy is computed from Gibbs' formula
\beq
S_B \,\doteq\, -k_B \int_{\qq} \int_{\pp} f(\pp,\qq,t) \log f(\pp,\qq,t) \, d\pp \, d\qq
\label{gibbs}
\eeq
where $\qq \in Q \subseteq \IR^n$ is a set of coordinates, and $\pp \in \IR^n$ are the corresponding conjugate momenta. The Lebesgue measure on the $2n$-dimensional phase space, $d\pp \, d\qq = dq_1 ... dq_n dp_1 ... dp_n$, is invariant under coordinate changes, though such changes do affect the bounds of integration $Q$, unless $Q = \IR^n$. The formula in (\ref{gibbs}) can be considered as a special case of (\ref{contentr}), to within the constant $k_B$.

For a statistical mechanical system at equilibrium the temporal dependence of $f$ vanishes and $f$ converges to
the Boltzmann distribution
\beq
f_{B}(\qq,\pp) \,\doteq\, \frac{1}{Z} \exp \left(-\beta {H}({\bf p},{\bf q}) \right)
\label{maxboltdist}
\eeq
where $\beta \doteq 1/k_B T$ with $k_B$ denoting Boltzmann's constant and $T$ is temperature measured in degrees Kelvin. Here the Hamiltonian of a mechanical system is defined as the total system energy written in terms the conjugate momenta $\pp$ and generalized coordinates $\qq$:
\beq
 H(\pp,\qq) \doteq \half \pp^T M^{-1}(\qq)\, \pp \, + \, V(\qq) \,.
\label{hamlsll4432}
\eeq
Here $M(\qq)$ is the configuration-dependent mass matrix and
$${\bf p} \,\doteq\, M(\qq) \dot{\qq} \,. $$
For recent results regarding the rate of approach to equilibrium, see \cite{entropy2021}.

But since a Gaussian integral has closed form, the marginal distribution resulting by computing the integral over ${\bf p}$ gives
\beq
\int_{\pp} f_{B}(\qq,\pp) d\pp = C \exp \left(-\beta {V}({\bf q}) \right) |M(\qq)|^{\half}
\label{maxboltdist1}
\eeq
where $|M(\qq)|$ denotes the determinant of $M(\qq)$ and
$C$ is the normalizing constant that makes the above expression a probability density with respect to the measure $d\qq$. It is natural to re-interpret this as
$$ f_C({\bf q}) \,\doteq\, C \exp \left(-\beta {V}({\bf q}) \right) $$
being the pdf with respect to the measure $|M(\qq)|^{\half} d\qq$.
This is the {\it configurational} Boltzmann distribution.
The associated {\it configurational entropy} is defined as
\beq
S_C \,\doteq\, -\int_{\qq} f_C({\bf q}) \log f_C({\bf q}) |M(\qq)|^{\half} d\qq \,.
\label{scdef}
\eeq
As an example, given a rotating body with orientation parameterized by ZXZ Euler angles,
${\bf q} = [\alpha,\beta,\gamma]^T$, it can be shown that $|M(\qq)|^{\half} = |I|^{\half} \sin \beta$
where $I$ is the moment of inertia tensor. The measure $\sin\beta d\alpha d\beta d\gamma$ is in fact the
bi-invariant integration measure for the rotation group $SO(3)$, and so (\ref{scdef}) can be viewed as
a specific coordinate-dependent version of
(\ref{Gent}).

\subsubsection{Measure-Theoretic Information and Entropy} \label{measureentropy}

Concepts used in information theory and statistical mechanics have been modified for use in the ergodic theory of deterministic
dynamical systems. In this section, the concept of measure-theoretic information, as it appears in ergodic theory, is reviewed.

Given any compact space on which a measure can be defined (for the sake of concreteness, think of a compact manifold with associated
volume element), it is possible to partition that space into a finite number of disjoint subsets, the union of which is, to within a set of measure
 zero, the whole space. That is, given a measurable space (e.g., a compact manifold), $M$, a partition $\alpha = \{A_i\}$ is defined such that
$$ A_{i_1} \cap A_{i_2} = \emptyset \hskip 0.2 true in {\rm if} \hskip 0.2 true in i_1 \neq i_2 \askip \bigcup_{i \in I} A_i = M. $$

The set indicator function,
\beq
I_A(x) \doteq \left\{\begin{array}{lll}
1  & {\rm if} & x \in A \\ \\
0 & {\rm if} & x \notin A
\end{array} \right.,
\label{probtrans53325}
\eeq
together with the partition $\alpha$ can be used to define {\it measure-theoretic information} as
\beq
I_\alpha(x) \doteq - \sum_{A \in \alpha} I_A(x) \log V(A)
\label{mastheent}
\eeq
where $V(A)$ is the volume of $A$ (or, more generally, the measure of $A$) normalized by the volume of the whole space, $M$. $I_\alpha(x)$ reflects the amount of ``information'' that results
from discovering that $x \in A$. If $A$ is a very large region then little information is gained by knowing that $x \in A$. Measure-preserving
actions of Lie groups such as rotations and translations do not affect this quantity.
If $\alpha$ is a
very fine partition, each subset of which has roughly the same volume, then more information is obtained from $I_\alpha(x)$ than if the
partition is coarse. Sometimes it is convenient to raise the subscript and write $I(\alpha)(x)$ in place of $I_\alpha(x)$.

{\it Measure-theoretic entropy} of a partition is defined as
\beq
H(\alpha) \doteq \sum_{A \in \alpha} z(V(A)). 
\label{meaent0302}
\eeq
where
\beq
z(\phi) \doteq \left\{\begin{array}{lll}
-\phi\log \phi  & {\rm if} & 0 < \phi \leq 1 \\ \\
0 & {\rm if} & \phi = 0
\end{array} \right. .
\label{zfuncdef}
\eeq
$H(\alpha)$ is related to $I(\alpha)(x)$ through the equality
$$ H(\alpha) = \int_M I(\alpha)(x) \,dx. $$

Cconcepts of conditional information and entropy have been articulated previously in this measure-theoretic
context, as reviewed in \cite{stochastic}.
Even more sophisticated concepts of entropy that are used in the description of dynamical systems have been built on this concept, including
the {\it Kolmogorov-Sinai entropy} \cite{st_sinaient} and the  {\it topological entropy} of Adler, Konheim and McAndrew \cite{st_adler}.
See, for example, \cite{st_manebook} for more details.
For further reading on general ergodic theory, see \cite{st_billingsley, Bunimovich, st_halmos, st_manebook, st_Moser, st_parry, st_petersen, erg_Ruelle, erg_Ulam}.
For works that emphasize the connection between ergodic theory and group theory, see \cite{erg_Kleinbock, erg_Margulis, st_Mooreerg, st_rokhlin, st_Templeman, st_walters}

\subsection{Entropy Differences, Relative Entropy, and Conditional Entropy} \label{entdifsec}

Regardless of the type of entropy chosen as a descriptive tool for modeling disorder, when it comes to
using entropy in the analysis of physical self-replicating systems, it is the comparison of entropy
in the unassembled state and the assembled state that is important. For example, if a single planar part is located at random in a box at random orientation, its continuous entropy may simply be the log of the volume
of the  space of all possible motions (translations and rotations), with translations limited by the range of
allowable motions of the center of the part. If in the assembly process the part is placed at a specific location, then the entropy difference is a measure of how challenging the task is. For example, if the placement
has some tolerance, then it is the log of the volume of space of motions tolerated in the final part location
that defines its entropy. The difference in entropy between ordered and disordered state is then the relevant measure of the difficulty of assembly, rather than the total entropy.

Discrete entropy is always nonnegative, whereas continuous entropy is not bounded from below. However, when
considering entropy differences, the line between discrete and continuous entropy can be blurred. For example,
discretizing the space of motions of a part at a length scale smaller than the tolerance of the part placement
in the final assemblage will mean that the difference of continuous entropies  and the difference of discretized
versions of continuous entropies (both in the disordered and assembled state) obtained by integrating over sufficiently small bins will be negligible. That is
\beq
\Delta S_{cont} \approx \Delta S_{disc} \,\geq\, 0
\eeq
for appropriate discretization.

A concept from information theory that is useful in proving theorems about the properties of entropy is that
of {\it relative entropy} (or Kullback-Leibler divergence). This can be defined either in the discrete
or continues case. For the continuous case it is
\beq
D_{KL}(p \,\|\, q) \,\doteq\, \int_{X} p(\xx) \log\left(\frac{p(\xx)}{q(\xx)}\right) d\mu(\xx) \,.
\eeq
In the discrete case the integral is replaced by a sum and the integration measure disappears.
Though $D_{KL}(p \,\|\, q) \neq S(q) - S(p)$, it should be noted that
$$ D_{KL}(p \,\|\, q) \,\geq 0 $$
in both the discrete and continuous cases.

Finally, it should be mentioned that the concept of {\it conditional entropy} has a role in characterizing
how sensory information can be used to reduce uncertainty during robotic assembly, including in self-replication
processes.

In the context of a measure space where ${\bf x} = [{\bf x}_1^T,{\bf x}_2^T]^T$,
the conditional probability density is defined as
$$ p({\bf x}_1 \,|\, {\bf x}_2) \,\doteq\, \frac{p({\bf x}_1,{\bf x}_2)}{p({\bf x}_2)}, $$
or in the event-based notation used earlier,
$$ p(E_1 \,|\, E_2) \,\doteq\, \frac{p(E_1 \cap E_2)}{p(E_2)} \,. $$
This is a pdf in the variable ${\bf x}_1$. If $p$ denotes $p({\bf x}_1,{\bf x}_2)$ and $q_2$ denotes the conditional
$p({\bf x}_1 \,|\, {\bf x}_2)$, then the
conditional entropy is the defined as
$$ S(p,q_2) \,\doteq\, \int_{X} p({\bf x}_1,{\bf x}_2) \log p({\bf x}_1 \,|\, {\bf x}_2) d{\bf x} \,. $$
It is related to entropy difference as
$$ S(p,q_2) - S(p,q_1) \,=\, S(p_1) - S(p_2) $$
where $p_i = p({\bf x}_i)$ is the marginal wherein all other degrees of freedom are integrated out.

Conditioning reduces entropy, and the concept of conditional entropy plays a role when information is gained
by sensing, as described in \cite{pk_partent}.

\section{Symmetry and Entropy} \label{symmetry}

This section focuses on the computation of entropy when there is a relationship between a probability density and a  symmetry group. Such a relationship can exist when the space over which the probability density is defined is a group, or if the symmetry group acts on a space and the pdf is invariant under the action. In order
to understand what this means, the first subsection reviews basic concepts in group theory.

\subsection{Symmetry of Parts and Assemblages}

Let $A$ be an assemblage of a constituent set of parts $\{B_i\}$. For example, the assemblage could be a self-replicating robot. The assemblage $A$ always will be of the form
\beq
A = \bigcup_i r_i B_i
\label{bigA}
\eeq
where $r_i \in SE(n)$ are the rigid-body transformations that describe how the parts are assembled.
Each part might have discrete or continuous rotational symmetries, and likewise for the assemblage as a whole.
For example, if the part or assemblage is a cube, there are 24 symmetry operations. Or if it is a cone, it has a one degree-of-freedom continuous group of rotational symmetries.
In the natural world, many viruses are assemblages of protein
parts that have icosahedral rotational symmetry with 60 elements, and the individual protein parts usually lack symmetry.

If $H$ denotes the symmetry group of $A$ then for every
${\bf x} \in A$ and $h \in H$ it is the case that $h \cdot {\bf x} \in A$ and as a whole we write
$$ h \cdot A = A \,\forall h \in H\,.$$
Similarly, if each $B_i$ has a symmetry group $K_i$ then
$$ k \cdot B_i = B_i \,\forall k \in K_i\,.$$

Consequently, if each $r_i$ in (\ref{bigA}) is replaced with $r'_i
= h r_i k$ we see that
$$ \bigcup_i r'_i B_i = h \bigcup_i r_i k B_i = h \bigcup_i r_i B_i = h \cdot A = A \,.$$
In other words, when there are symmetries many ways exist to form an assemblage,
and this makes the job easier.

The space of possible spatial relationships of $N$ parts
prior to the formation of the assemblage is the direct product group $G = SE(n)^N$.
The distribution of parts can be described
by the joint distribution $f(g_1,g_2,...,g_N)$. When the parts are far from each other, this can be approximated
as a product of independent pdfs $f_i(g_i)$. But when they are being assembled, this is not true. Regardless, what
is true is that $f$ will inherit the symmetries of $A$ and each $B_i$, and this will impact the overall parts
entropy. This observation is quantified in the following sections.

\subsection{Groups and Coset Spaces}

Group theory is one of the foundational pillars of abstract algebra, and therefore of all of modern mathematics.
The concept of a group is simple. A group is a set, $G$, together with an operation, $\circ$, such that four properties hold: 1) for every pair $g_1, g_2 \in G$, the product $g_1 \circ g_2 \in G$; 2) a special element $e \in G$ exists such that $g \circ e = e \circ g = g$ for every $g \in G$; 3) For each $g \in G$ there exists an element
$g^{-1} \in G$ such that $g \circ g^{-1} = g^{-1} \circ g = e$; 4) For any three elements $g_1, g_2, g_3 \in G$
the associative property $(g_1 \circ g_2) \circ g_3 = g_1 \circ (g_2 \circ g_3)$ holds. The group is then referred
to as $(G,\circ)$, or when $\circ$ is understood, then the group is denoted simply as $G$ and the product
of $g_1$ and $g_2$ is $g_1 g_2$ rather than  $g_1 \circ g_2$.

The set $G$ can be discrete like the integers, or it can be a continuous (measurable) space. If it is discrete it can have either a finite or infinite number of elements. If it is continuous it can have finite or infinite volume, like the group of rotations or translations of Euclidean space, respectively. The continuous groups with added conditions on the analyticity of the group product and inversion defines Lie groups.

In practice in engineering most groups encountered are either finite-dimensional matrix Lie groups or their discrete
subgroups. For example, the group of translations of the real line $(\IR,+)$ can be described by $2\times 2$ matrices and the group product can be described as multiplication as
$$
\left(\begin{array}{cc}
1 & x \\
0 & 1
\end{array}\right)
\left(\begin{array}{cc}
1 & y \\
0 & 1
\end{array}\right)
\,=\,
\left(\begin{array}{cc}
1 & x+y \\
0 & 1
\end{array}\right)\,. $$

A subgroup is a set inside of a group which contains the identity and
is closed under: 1) multiplication of its elements; and
2) inversion of elements. The associative property holds automatically.
For example, the integers are a subset of the real numbers,
$\mathbb{Z} \subset \IR$ and the group of integers is a subset of the group of real numbers
under addition $(\mathbb{Z},+) \,<\, (\IR,+)$. (The notation $<$ is used to denote subgroup.)

In general if $H < G$, then $G$ can be partitioned into left or right cosets. For any given $g \in G$,
the left coset containing $G$ is defined as
$$ g H \,\doteq\, \{g h\,|\, h \in H\} $$
and a right coset is defined as
$$ H g \,\doteq\, \{h g\,|\, h \in H\} \,. $$
In general $gH \neq Hg$, but equality can hold in special cases.

It is possible for two different elements $g_1, g_2 \in G$ to generate the same coset, e.g.,
$g_1 H = g_2 H$. If all distinct cosets are collected, the result is called a coset space.
The left coset space is
$$ G/H \,=\, \{gH \,|\,  g \in G\} \,. $$
If $H$ and $G$ have a finite number of elements, then the number of elements in $G/H$ is
$|G/H| = |G|/|H|$. This is Lagrange's theorem, which has been known for more than 200 years.
When $G$ is a Lie group and $H$ is a discrete subgroup, the resulting coset space will have the same dimension as $G$, otherwise the coset space will have dimension that is the difference in the dimension of $G$ and $H$.

\subsection{Lie Groups}

Lie groups are groups wherein the set $G$ is a differentiable manifold and the operations of group multiplication
and inversion are analytic. The group $(\IR,+)$ is an example of a Lie group.

The group of rigid-body displacements in the Euclidean plane, $SE(2)$, can be described with
elements of the form
\beq
g(x,y,\theta) \,=\, \left(\begin{array}{ccc}
\cos \theta & -\sin\theta & x \\
\sin\theta & \cos\theta & y \\
0 & 0 & 1 \end{array}\right) \,.
\label{se2elem}
\eeq
The dimension is 3 because there are three free parameters $(x,y,\theta)$. This group is not compact as
$x$ and $y$ can take values on the real line.

The group of pure rotations in 3D can be described by rotation matrices
$$ SO(3) \,\doteq\, \{R \in \IR^{3\times 3} \,|\, RR^T = \II\,,\, \det R = +1\}\,. $$
$SO(3)$ is a compact 3-dimensional manifold. Again, the fact that the dimension of the matrices is also 3 is
coincidental.

A unimodular Lie group is defined by the property that a measure $dg$ can be constructed
such that the integral over the group has the property that
\beq
\int_G f(g)\, dg \,=\, \int_G f(g_0 g)\, dg
\,=\, \int_G f(g g_0)\, dg
\label{shiftinv}
\eeq
for any fixed $g_0 \in G$ and any function $f \in L^1(G)$.  It can also be shown that as a consequence of (\ref{shiftinv})
\beq
\int_G f(g)\, dg \,=\,  \int_G f(g^{-1})\, dg  \,.
\eeq
These properties are natural generalizations of those familiar to us for functions on Euclidean space.

As we are primarily concerned with
probability density functions for which
$$ \int_G f(g) dg = 1\,, $$
these clearly meet the condition of being in $L^1(G)$.

In the case of $SO(3)$ the bi-invariant measure expressed in terms of $Z-X-Z$ Euler angles
$(\alpha,\beta, \gamma)$ is $dR = \sin\beta d\alpha d\beta d\gamma$. In the case of $SE(2)$,
the bi-invariant measure is $dg = dx dy d\theta$.

The convolution of probability density functions on a unimodular Lie group is a natural operation,
and is defined as
\beq
(f_1 * f_2)(g) \,\doteq\, \int_{G} f_1(h) f_2(h^{-1} g) \, dh \,.
\label{convdef}
\eeq
The convolution of two probability density functions is again a probability density.

\subsection{Entropy and Group-Theoretic Decompositions}

Aside from the ability to sustain the concept of convolution, one of the fundamental ways that groups
resemble Euclidean space is the way in which they can be decomposed. In analogy with the way that
an integral over a vector-valued function with argument $\xx \in \IR^n$ can be decomposed into integrals
over each coordinate, integrals over Lie groups can also be decomposed in natural ways. This has
implications with regard to inequalities involving the entropy of pdfs on Lie groups. Analogous expressions
hold for finite groups, with volume replaced by the number of group elements.

\subsubsection{Decomposition of Integrals over Subgroups and Coset Spaces}

Much like the way an intergral over $\IR^2$ can be partitioned into integrals over two copies of $\IR$,
integrals of functions on groups can be partitioned into integrals over subgroups and corresponding coset spaces
In particular, it is possible to retain one representive element of each coset $G/H$ to form a fundamental domain
$F_{G/H} \subset G$. Then the integral of any measurable function can be decomposed as \cite{harmonic}
\begin{equation}
\int_{G} f(g)\,dg =
\int_{F_{G/H}} \left(
\int_H f(g \circ h)\,dh \right)\,d(g)
\label{quomeas}
\end{equation}
where $d(g)$ and $dh$ are
the invariant integration measures
on $F_{G/H}$ and $H$. If $H$ is a discrete subgroup then the integral over $H$ becomes a summation and
$F_{G/H}$ can be taken to be a region with the same dimension as $G$ and $d(g) = dg$.
If $H$ is a Lie subgroup of Lie group $G$, then $F_{G/H}$ is a lower dimensional subset of $G$ and
the intepretation of the integral with respect to $d(g)$ is that it is $dg$ restricted to a lower dimensional space, much like how the integration measure for a surface in Euclidean space is induced from the
ambient Lebesgue measure.

Given a probability density function $f_G : G \,\longrightarrow\, \mathbb{R}_{\geq 0}$, which by definition
satisfies the condition
$$ \int_{G} f_G(g)\,dg \,=\,1 \,, $$
we can define
\begin{equation}
f_{G/H}(gH) \,\doteq\, \int_H f_G(g \circ h)\,dh
\,\,\,\,\, {\rm and}\,\,\,\,\, f_{H}(h) \,\doteq\, \int_{F_{G/H}} f_G(g \circ h)\,d(g)
\,,
\label{ftildedef}
\end{equation}
where $h \in H$.

The coset space $G/H$ is not a subset of $G$, but there is a natural map from $G$ to $G/H$ defined
by $g \,\longrightarrow\, gH$. The integration measure on $G/H$ is defined by the equality
$$ \int_{G/H} f_{G/H}(gH)\,d(gH)
\,\doteq\, \int_{F_{G/H}} f_{G/H}(gH) d(g) \,. $$

Then $F_{G/H}$ and $f_H$ are bona fide probability density functions because
\begin{equation}
\int_{G} f_G(g)\,dg =
\int_{G/H} \left(
\int_H f_G (g\circ h)\,dh \right)\,d(gH) \,=\, \int_{G/H} f_{G/H}(gH)\,d(gH)
\label{quomeas2}
\end{equation}
and similarly
\begin{equation}
\int_{G} f_G(g)\,dg = \int_H
\left(\int_{F_{G/H}}
 f_G (g\circ h)\,d(g) \right)\,dh \,=\, \int_{H} f_{H}(h)\,dh \,.
\label{quomeas2k2kj2e}
\end{equation}

The same construction can be made for right cosets. Moreover,
if we are given a unimodular group $G$ with unimodular
subgroups $K$ and $H$ such that $K \leq H$, we can use
the facts that
$$
\int_{G} f_G(g)\,d(g) =
\int_{G/H} \int_H f_G(g \circ h)\,dh\,d(gH),
$$
and
$$
\int_{H} f_H(h)\,dh =
\int_{H/K} \int_K f_H(h \circ k)\,dk\,d(hK) $$
to decompose the integral
of any nice function $f(g)$  as
\begin{equation}
\int_{G} f(g)\,d(g) =
\int_{G/H} \int_{H/K} \int_{K}
f(g \circ h \circ k)\,dh\,d(hK)\,d(gH).
\label{decomp67}
\end{equation}

The integral of a function
on a group can also be decomposed in terms of two arbitrary subgroups
and a double coset space as
$$ \int_{G} f(g)\,d(g) = \int_{K} \int_{K\backslash G/H}
\int_{H} f(k \circ g \circ h)\,dh \,d(KgH)\,dk. $$
This can be realized using the concept of fundamental domains in $G$ constructed from one representative
element of $G$ per double coset. For concrete examples of fundamental domains of coset and double-coset spaces where $G = SO(3)$ and $H$ and $K$ are finite subgroups, see \cite{Wuelker}.

\subsubsection{Entropy and Group-Theoretic Decompositions}

Here some theorems related to entropy and the decomposition of integrals on groups
originally derived in \cite{gc-geomech,stochastic} are reviewed. These will be relevant later when considering
entropy differences between disordered and assembled states of parts with symmetry.

\begin{theorem} \label{th5.4}
The entropy of a pdf on a unimodular Lie group is no greater than the sum of the marginal
entropies on a subgroup and the corresponding coset space:
\beq
S(f_G) \leq S(f_{G/H}) + S(f_H).
\label{skcnlrt543d}
\eeq
\end{theorem}
\begin{proof}
This inequality follows immediately from the nonnegativity of the Kullback-Leibler divergence
$$ D_{KL}(f_G \, \| \, f_{G/H} \cdot f_H) \geq 0. $$
\end{proof}

For example, if $G=SE(n)$ is a motion group of $n$-dimensional Euclidean space consisting of
rotation-translation pairs of the form $(R,{\bf t})$ and group law
$$ (R_1, {\bf t}_1) \circ (R_2, {\bf t}_2) = (R_1 R_2, R_1 {\bf t}_2 + {\bf t}_1) $$
and if $H \cong \IR^n$ is the subgroup of pure translations of the form $(\II, {\bf t})$
in $n$-dimensional Euclidean space, then $G/H \cong SO(n)$ consisting of all elements of the form
$(R,{\bf 0})$, and an arbitrary element of $SE(n)$ is written as a pair $(R,{\bf t}) \in SO(n) \times \IR^n$,
then $SO(n) \cong SE(n)/\IR^n$ and we can write
\bea
\int_{SE(n)} f(g) \,d(g) \,&=&\, \int_{SO(n)} \int_{\IR^n}  f(R,{\bf t}) \,d{\bf t} \,dR  \\
\,&=&\, \int_{SO(n)} \left(\int_{\IR^n} f ((\II,{\bf t})\circ (R,{\bf 0})) \,d{\bf t} \right)\,dR \,,
\eea
and the marginal entropies on the right-hand-side of (\ref{skcnlrt543d}) are those computed for pure rotations and pure translations.

\begin{theorem} \label{th5.5}
The entropy of a pdf on a group is no greater than the sum of marginal entropies over any two subgroups and the corresponding double-coset space:
\beq
S(f_G) \leq S(f_K) + S(f_{K\backslash G/H}) + S(f_H).
\label{skcnlrt543jlvd}
\eeq
\end{theorem}
\begin{proof}
Let
$$ f_K(k) = \int_{K\backslash G/H} \int_H f_G(k \circ c_{K\backslash G/H}(KgH) \circ h) \,dh\, d(KgH) $$
$$ f_H(h) = \int_{K\backslash G/H} \int_K  f_G(k \circ c_{K\backslash G/H}(KgH) \circ h) \,dk\, d(KgH)  $$
and
$$ f_{K\backslash G/H}(KgH) = \int_K \int_H f_G(k \circ c_{K\backslash G/H}(KgH) \circ h)\, dh\, dk, $$
where $c_{K\backslash G/H}: K\backslash G/H\,\rightarrow\,G$ is a function that selects an element of $G$ from
each double coset $KgH$ to form a fundamental domain $F_{K\backslash G/H} \subset G$.
Then again using the nonnegativity of the Kullback-Leibler divergence
$$ D_{KL}(f_G \, \| \, f_K \cdot f_{K\backslash G/H} \cdot f_H) \geq 0 $$
gives (\ref{skcnlrt543jlvd}).
\end{proof}

\begin{theorem} \label{th5.6}
The entropy of a pdf is no greater than the sum of entropies of its marginals over coset spaces
defined by nested subgroups $H < K < G$:
\beq
S(f_G) \leq S(f_{G/K}) + S(f_{K/H}) + S(f_H).
\label{skcnlrt543dhh}
\eeq
\end{theorem}
\begin{proof}
Given a subgroup $K$ of $H$, which is itself a subgroup of $G$ (that is, $H < K < G$), apply (\ref{skcnlrt543d}) twice.
Then $S(f_G) \leq S(f_{G/K}) + S(f_{K})$ and
$S(f_K) \leq S(f_{K/H}) + S(f_H)$, resulting in (\ref{skcnlrt543dhh}).
Explicitly, $g = c_{G/K}(gK) \circ c_{K/H}(kH) \circ h$, and so $f_G(g) = f_G(c_{G/K}(gK) \circ c_{K/H}(kH) \circ h)$. Therefore,
$$ f_{G/K}(gK) = \int_{K/H} \int_H f_G(c_{G/K}(gK) \circ c_{K/H}(kH) \circ h) \,dh \,d(kH) $$
$$ f_{K/H}(kH) = \int_{G/K} \int_H f_G(c_{G/K}(gK) \circ c_{K/H}(kH) \circ h) \,dh\,  d(gK) $$
and
$$ f_H(h) = \int_{G/K} \int_{K/H} f_G(c_{G/K}(gK) \circ c_{K/H}(kH) \circ h) \,d(kH)\, d(gK). $$
\end{proof}

\subsection{Application to Self-Replicating Robots}

As stated earlier, the entropy difference between a disordered state and an assembled one
is a measure of the difficulty of the assembly process. In a maximally disordered state, the entropy
might simply be the log of the volume of allowable motions, which is the maximal entropy possible for
an isolated part confined to a finite volume. Or, if an external potential field is applied,
the probability density would become the configurational Boltzmann distribution, which has lower entropy,
making an assembly process easier. Another way to reduce entropy is by reducing the size of the free space
in which parts can move. This is discussed in detain in the next section.

But the relevance of the formulations earlier in this section are not related to the entropy of the
disordered state. Rather, the role of part symmetry is related to the allowed entropy of the assembled state.
For example, if a part is a cube that needs to be inserted into an assemblage, then the space of allowable correct
orientations of the part in the final assemblage is 24 times larger than if no symmetry existed. Or, in an extreme
case, if the part is spherical, and it only needs to be inserted into a hemispherical slot, then orientation is
completely irrelevant. For example, a ball on a roulette wheel easily finds its place. Another example would be
loading bullets into a revolver wherein the $SO(2)$ symmetry of the bullets and the cylindrical nature of the
chambers make it much easier than if the cross sections had no symmetry. In automated manufacturing and assembly systems these principles are well known \cite{pk_boothroyd, pk_booth1, pk_Homem, pk_erdmann,pk_liu,pk_Whitney}.

The value added by the analysis of this section is that the entropy of the assembled state can be quantified
when constituent parts (and possibly the assemblage as a whole) have symmetry. If an individual part has symmetry
group $K$, then the space of motions that need to be considered during assembly is reduced from $G$ to $G/K$.
Or, put another way, if $K$ is a finite subgroup there are $|K|$ times as many correct ways to assemble in comparison to when there is no symmetry. Similarly, if the overall assemblage has symmetry group $H$, then the space is reduced from $G$ to $H\backslash G$. And when both the assemblage and constituent parts have symmetry then
the reduction is from $G$ to $H\backslash G/K$. Any and all such symmetries effectively reduce the entropic burden of the assembly process.

Explicitly, if a part has symmetry then its configurational probability density function in the assemblage
will inherit this symmetry as $f(g) = f(g \circ k)$ for all $k \in K$, making it a left-coset function for which
the analysis in the previous subsections becomes directly applicable. Or, put another way
$$ f(g) = \frac{1}{|K|} \sum_{k\in K} f(g \circ k) \,. $$
In the case of a continuous symmetry the sum is replaced by an integral, and $|K|$ is replaced by the volume of $K$.
A general property of entropy, which follows from the theory of convex functions, is that averaging of any sort
increases entropy.

\section{Parts Entropy and the Principal Kinematic Formula} \label{kinematic}

Physical self-replicating systems that are able to assemble basic parts to form replicas of themselves in the presence of uncertainties in the positions and orientations of feed parts are more robust than those that require perfectly palletized input parts. As discussed earlier, ``parts entropy'' is a statistical measure of the ensemble of all possible positions and orientations of a single part with a given probability density in position and orientation. Here a related issue is considered: if the part is confined to move uniformly at random in a finite container, what is its parts entropy ?

In this section it is shown how the  ``Principal Kinematic Formula'' (PKF) from the field of Integral Geometry can be used
to model the reduction in allowable motion imposed by the presence of an obstacle. Since entropy in the disordered state is related to the volume of allowable motion, this is relevant to the analysis of entropy change in
self-replicating systems. Here the PKF is stated without proof. References on this topic in which derivations are provided include \cite{pk_santalo}.

\subsection{The Principal Kinematic Formula for Collision}

The {\it indicator function} on any measurable body, $C$, is defined by:
$$ \iota(C) \,\doteq\, \left\{\begin{array}{cc}
               1 & \mbox{if $C \neq \emptysett$} \\
               0 & \mbox{for $C = \emptysett$}
               \end{array} \right. $$
If $g\in G$ is an element of a group (e.g., the group of rigid-body motions, $SE(n)$)
that acts on $C$ without shrinking
it to the empty set, then $\iota(gC) = \iota(C)$ where
$$ g C \,\doteq\, \{g \cdot {\bf x}| {\bf x} \in C\}. $$
For now let $G = SE(n)$, the group of rigid-body motions in $\IR^n$.
If $g = (A, {\bf a})$ is the rigid-body motion with rotational part
$A \in SO(n)$ and translational part ${\bf a} \in \IR^n$, then the action
of $G$ on $\IR^n$ is $g \cdot {\bf x} = A {\bf x} + {\bf a}$, and hence $gC$ is well defined.
The indicator function is one of many functions on a body that is invariant under rigid-body motion. Others
include the volume of the body, the surface area (or perimeter in the two-dimensional case).

Given two convex bodies, $C_0$ and $C_1$.
Let $C_0$ be stationary, and let $C_1$ be mobile.
The intersection of these two convex bodies is either a
convex body or is empty.
Furthermore, the rigid-body motion (or even affine deformation)
of a convex body does not change the fact that it is convex. Therefore,
when $C_0 \cap gC_1$ is not empty it will be a convex body, and
$$f_{C_0,C_1}(g) \,\doteq\, \iota(C_0 \cap gC_1) $$
will be a compactly supported function on $G$ that takes the value of $1$ when $C_0$ and the moved
version of $C_1$ (denoted as $gC_1$) intersect, and it will be zero otherwise.

Counting up all values of $g$ for which an intersection occurs is then
equivalent to computing the integral
\begin{equation}
{\mathcal I}(C_0, C_1) \,=\, \int_{G} \iota(C_0 \cap g \cdot C_1) \,dg.
\label{jdef}
\end{equation}

An amazing result is that the integral ${\mathcal I}$ can be computed exactly using only elementary geometric properties of the bodies $C_0$ and $C_1$ without actually having
to perform an integration over $G$. While the general theory has been developed
by mathematicians for the case of bodies in $\IR^n$ \cite{pk_chern} and in manifolds
on which some Lie group acts (see \cite{pk_santalo} and references therein), we are concerned only with the cases
of bodies in $\IR^2$ and $\IR^3$.

In the planar case, we can write (\ref{jdef}) explicitly as
\begin{equation}
{\mathcal I}(C_0, C_1) = \int_{-\pi}^{\pi} \int_{-\infty}^{\infty} \int_{-\infty}^{\infty}
\iota(C_0 \cap g(t_1,t_2,\theta)C_1) \,dt_1 dt_2 d\theta
\label{jdefse2}
\end{equation}
where the rotational part of $g = (R,{\bf t})$ is described by
$$ R = \left(\begin{array}{cc}
\cos \theta & -\sin \theta \\
\sin \theta & \cos \theta \end{array} \right) $$
and the translational part is given by the vector ${\bf t} = [t_1, t_2]^T$.

Spatial rigid-body motions can be parameterized as
$$ g(t_1,t_2,t_3; \alpha, \beta, \gamma) = \left(\begin{array}{ccc }
R(\alpha,\beta,\gamma) & & {\bf t} \\ \\
{\bf 0}^T & & 1 \end{array} \right), $$
where $R(\alpha,\beta,\gamma)$ is a rotation matrix expressed in terms of the ZXZ Euler-angles and ${\bf t} \in \IR^3$ is the translation vector. The bi-invariant integration measure for the group $SE(3)$ is then, to within an arbitrary scaling constant,
$$ dg \,=\, \sin \beta \, d\alpha \, d\beta \, d\gamma \, dt_1 \, dt_2 \, dt_3. $$

\begin{theorem}
(Blaschke, \cite{pk_blaschke2}):
Given planar convex bodies $C_0$ and $C_1$, with $C_0$ fixed and $C_1$ free to move under the action
of $SE(2)$, then the volume in the region of $SE(2)$ that places the bodies in collision is
\begin{equation}
{\mathcal I}(C_0, C_1) = 2\pi [A(C_0) + A(C_1)] + L(C_0) L(C_1)
\label{pbl1}
\end{equation}
where $L(\cdot)$ is the perimeter of a body, and $A(\cdot)$ is the area.
\end{theorem}
For example, if the bodies are disks of radius $r_0$ and $r_1$, the above formula gives
$$ {\mathcal I}(C_0, C_1) = 2\pi [\pi r_0^2 + \pi r_1^2] + (2\pi r_0)(2\pi r_1)\,. $$
Clearly for disks the condition for collision is that the distance between the centers is less than or equal to
$r_0+r_1$, and so for this example
$$ {\mathcal I}(C_0, C_1) = (2\pi) \cdot \pi (r_0+r_1)^2 $$
where $2\pi$ is the volume of the space of planar rotations, $SO(2)$. The above two expressions are equal for disks.

The three-dimensional analog of this formula is given in the theorem below.

\begin{theorem}
(Blaschke, \cite{pk_blaschke2}):
Given 3D convex bodies $C_0$ and $C_1$, then when holding $C_0$ fixed and allowing $C_1$ to move, the volume
in the space of motions corresponding to the bodies being in collision will be
\beq
{\mathcal I}(C_0, C_1) \,=\, 8\pi^2[V(C_0)+V(C_1)] \label{sbl1}
+ 2\pi [A(C_0) M(C_1) + A(C_1) M(C_0)]
\eeq
where $V(\cdot)$ is the volume of the body and
$M(\cdot)$ and $A(\cdot)$ are respectively the integral of mean curvature
area and of the surface enclosing a body.
\end{theorem}
This result was developed by Wilhelm Blasche a century ago, and proofs can be found in
\cite{stochastic,pk_blaschke2, pk_santalo}.

Note that ${\mathcal I}(C_0, C_1) \,=\, {\mathcal I}(C_1, C_0)$, which is a consequence of the bi-invariance of
integration on the Lie group $SE(3)$, which is unimodular.

As an example of the above theorem, if $C_i$ is a solid ball of radius $r_i$, then the above formula gives
$$ {\mathcal I}(C_0, C_1) \,=\, 8\pi^2 \cdot \frac{4\pi}{3} [r_0^3 + r_1^3]
+ 2\pi [4\pi r_0^2 \cdot 4\pi r_1 + 4\pi r_1^2 \cdot 4\pi r_0] \,. $$
This matches the expected result of
$$ {\mathcal I}(C_0, C_1) \,=\, 8\pi^2 \cdot \frac{4\pi}{3} (r_0 + r_1)^3 $$
which is the volume of $SO(3)$ multiplying the volume of the ball of radius $r_0+r_1$
corresponding to the volume of all motions of the center of $C_1$ that would place it in collision with $C_0$.

The literature on integral geometry spanning the past century is immense. For further reading on the Principal Kinematic Formula (and Integral Geometry more generally) see \cite{pk_Klain,pk_poin,pk_santalo,SchneiderWeilbook}.
Related work is concerned with determining when
one body can be contained in another \cite{pk_Zhang,pk_zhouj92,pk_zhouj95,pk_zhouj98} and the characterization
of free motion of one body moving inside another \cite{pk_Karnik,kin-contain}.
Namely, if a part can be contained, what is the volume of its free motion ? This is obviously related
to the parts entropy of a final assemblage in which there are clearances.
This is the subject of the following section.

\subsection{Kinematics of Containment}

Instead of considering the volume within the motion group describing collision of two bodies, we can instead evaluate the volume of allowable motion of a small body within a large convex container. In this scenario a formula similar to the principal kinematic formula results under mild conditions. Namely, if all principal curvatures of the
inner body are larger than every principal curvature of the container, then in the planar case \cite{kin-contain}
\begin{equation}
{\mathcal V}(C_1, C_2) = 2\pi [A(C_1) + A(C_2)] - L(C_1) L(C_2)
\label{pbl1a}
\end{equation}
and in the spatial case
\beq
{\mathcal V}(C_1, C_2) \,=\, 8\pi^2[V(C_1)+V(C_2)] \label{sbl1a}
- 2\pi [A(C_1) M(C_2) + A(C_2) M(C_1)]
\eeq
Where $C_2$ is the container and $C_1$ is again the moving body

\subsection{Entropy of a Convex Part Free to Move in a Container with an Obstacle}

Suppose that as a strategy for assembling parts in a robotic self-replication process, the original
robot first either cages a part that is to be assembled, or pushes it into a bowl. These actions
do not require sophisticated manipulation, and hence are appropriate for simplifying the requirements on
a self-replicating system. Both a cage and bowl
are examples of containers. This process physically
limits the parts entropy.

If there no preferred positions and orientations of the part within the container, then the parts entropy
is simply the volume of allowable motion. In the case when there is an obstacle (such as a post or pillar)
in the container that limits allowable motion, and if the obstacle and container geometries are such that
the moving part never gets jammed between the container and the obstacle, the parts entropy will be
$$ S \,=\, \log \left({\mathcal V}(C_1, C_2)  - {\mathcal I}(C_0, C_1) \right) \,. $$

\section{Error Propagation in Parts Manufacturing} \label{error}

When considering the production of parts from raw materials, the question of how to maintain fidelity generation after generation arises. The process of reliably producing replicas
by employing error-reduction techniques is related to information theoretic entropy and the theory of
error-correcting codes. One way to achieve error correction is to use symmetry as a parity check. That is, original parts that are intended to be symmetrical in the first generation will lose symmetry in future generations.
For example, if a mould makes a casting and that casting is used to make a subsequent mold, then with each iteration
there will be some corruption of the result due to random flaws that are introduced. But if the parts have symmetries, then random flaws will affect different areas of a part in different ways, and imposing information about the symmetry of the original parts on the replicas as a post-processing step in the manufacturing process will
help in maintaining tolerances and to squash the magnitude of errors in reproduction.

No matter what error correction methods are put in place, changes resulting from compounded manufacturing errors in artificial self-replicating machines would almost always result in reduction of functionality over generations.
This might not be a bad thing if the purpose of a self-replicating physical system is to magnify initial human effort in harvesting in situ resources in outer space. For example, a magnification by a factor between 10-100 before they cease to replicate would be tremendous, whereas developing immortal (and evolvable) self-replicating robots could have unforseen negative consequences.

Whereas living systems have various levels of feedback to ensure stability over generations from the molecular (DNA replication and transcription) to the macroscopic (competitive survival advantages and preditor-prey equilibria), the same is not true in scenarios in which artificial self-replicating systems would be deployed. For example, a self-replicating factory designed to harvest materials on the moon or Mars or other inert environments in order to bootstrap mankind's reach into the solar system need not worry about competing for resources with other entities. The goal of self-replicating systems is not evolution, but rather the reliable
copying of engineered systems that have a specific mission beyond survival.


\section{Conclusions} \label{conclusions}

Artificial self-replicating systems that process materials, make parts, and assemble the parts to make physical copies of themselves have been imagined for more than half a century. The potential impact of such systems for the development of resources in outer space are tremendous. Indeed, this might be the only path forward for colonization of the solar system, as well as harnessing resources to mitigate the effects of global warming. Various technologies have developed in recent years that make self-replication with vitamin resources more realistic than ever. These include metal additive manufacturing (a.k.a 3D printing), and new methods for materials processing. The purpose
of this paper was not to review these technological advances, but rather to expand the set of theoretical and algorithmic foundations with which to evaluate progress in this field. Several topics that interwine entropy,
information, complexity, and error propagation were articulated. Advanced mathematical tools from the theory of
Lie groups and Integral Geometry were introduced to the a-life community in the context of artificial physical
self-replicating systems.


\begin{thebibliography}{99}

\bibitem{f1}
Freitas, R.A. and Merkle, R.C., 2004. Kinematic self-replicating machines. Landes.

\bibitem{f2}
Freitas, R.A. and Gilbreath, W.P., 1982. Advanced automation for space missions. Journal of the Astronautical Sciences, 30(1), p.221.

\bibitem{f3}
Freitas, R., Zachary, W., 1981. A self-replicating, growing lunar factory. In 4th Space manufacturing; Proceedings of the Fifth Conference (p. 3226).

\bibitem{c1}
Chirikjian, G.S., Zhou, Y. and Suthakorn, J., 2002. Self-replicating robots for lunar development. IEEE/ASME transactions on mechatronics, 7(4), pp.462-472.

\bibitem{l1}
Lee, K., Moses, M. and Chirikjian, G.S., 2008. Robotic self-replication in structured environments: Physical demonstrations and complexity measures. The International Journal of Robotics Research, 27(3-4), pp.387-401.

\bibitem{l2}
Lee, K. and Chirikjian, G.S., 2007. Robotic self-replication. IEEE robotics and automation magazine, 14(4).

\bibitem{moses2020}
Moses, M.S. and Chirikjian, G.S., 2020. Robotic self-replication. Annual Review of Control, Robotics, and Autonomous Systems, 3, pp.1-24.

\bibitem{pk_case}
Chirikjian, G.S., ``Parts Entropy, Symmetry, and the Difficulty
of Self-Replication,'' {\it Proc. ASME Dynamic Systems and Control
Conference}, Ann Arbor, Michigan, Oct 20-22, 2008.

\bibitem{pk_dss}
Chirikjian, G.S., ``Parts Entropy and the Principal Kinematic Formula,'' {\it Proc. IEEE Conference on Automation Science and Engineering}, pp. 864--869, Washington D.C., August 23-26, 2008.


\bibitem{Zurek}
Zurek, W.H., 2018. Complexity, entropy and the physics of information. CRC Press.

\bibitem{Pincus}
Pincus, S.M., 1991. Approximate entropy as a measure of system complexity. Proceedings of the National Academy of Sciences, 88(6), pp.2297-2301.

\bibitem{Teix}
Teixeira, A., Matos, A., Souto, A. and Antunes, L., 2011. Entropy measures vs. Kolmogorov complexity. Entropy, 13(3), pp.595-611.

\bibitem{Zur2}
Zurek, W.H., 1991. Algorithmic Information Content, Church—Turing Thesis, Physical Entropy, and Maxwell’s Demon. In Information Dynamics (pp. 245-259). Springer, Boston, MA.

\bibitem{Zurek2}
Zurek, W.H., ``Thermodynamic Cost of Computation, Algorithmic Complexity, and the Information Metric,'' {\it Nature} 341:119--124, 1989.

\bibitem{Benn}
Bennett, C.H., 1982. The Thermodynamics of Computation--a Review. International Journal of Theoretical Physics, 21(12).

\bibitem{Sz}
Szilard, L., 1964. On the decrease of entropy in a thermodynamic system by the intervention of intelligent beings. Behavioral Science, 9(4), pp.301-310.

\bibitem{Br}
Brillouin, L., {\it Science and Information Theory}, $2^{nd}$ ed., Academic Press, New York, 1962.

\bibitem{Land1}
Landauer, R., 1996. The physical nature of information. Physics letters A, 217(4-5), pp.188-193.

\bibitem{Land2}
Landauer, R., 1991. Information is physical. Physics Today, 44(5), pp.23-29.

\bibitem{st_leff1}
Leff, H.S., Rex, A.F., {\it Maxwell's Demon: Entropy, Information, Computing}, Princeton University Press, Princeton NJ, 1990.

\bibitem{st_leff2}
Leff, H.S., Rex, A.F., {\it Maxwell's Demon 2: Entropy, Classical and Quantum Information, Computing}, Institute of Physics Publishing, Bristol and Philadelphia, 2003.

\bibitem{Zurek1}
Zurek, W.H. ed., {\it Complexity, Entropy and the Physics of Information}, Sante Fe Institute Studies in the Sciences of Complexity, Vol. 8, Addison-Wesley,
Reading Mass, 1990.


%

\bibitem{menezes}
Menezes AA, Kabamba PT.,
``Optimal seeding of self-reproducing systems,'' {\it Artificial Life} 18:27–51,  2011.

\bibitem{sayama}
Sayama H., ``Construction theory, self-replication, and the halting problem,''
{\it Complexity}, 13:16–22, 2008.

\bibitem{pk_partent}
Sanderson, A.C., ``Parts Entropy Methods For Robotic Assembly System
Design" , {\it Proceedings of the 1984 IEEE International Conference
on Robotics and Automation (ICRA '84)}, Vol. 1, pp. 600 - 608, March
1984.

\bibitem{entropy2021}
Chirikjian, G.S., ``Rate of Entropy Production in Stochastic Mechanical Systems,''
{\it Entropy} 2022, 24, 19. https://doi.org/10.3390/e24010019


\bibitem{harmonic}
Chirikjian, G.S., Kyatkin, A.B., {\it Harmonic Analysis for Engineers and Applied Scientists},
Dover, Mineola, NY, 2016.


\bibitem{Wuelker}
Wuelker, C., Ruan, S. and Chirikjian, G.S.,
``Quantizing Euclidean motions via double-coset decomposition.'' {\it Research}, 2019.

\bibitem{gc-geomech}
Chirikjian, G.S., 2010. Information-theoretic inequalities on unimodular Lie groups. Journal of geometric mechanics, 2(2), p.119.


\bibitem{stochastic}
Chirikjian, G.S.,  {\it Stochastic Models, Information Theory, and Lie Groups: Volumes I + II},
Birkh\"{a}user, Boston, 2009/2012.





\bibitem{st_adler}
Adler, R.L., Konheim, A.G., McAndrew, M.H., ``Topological Entropy,'' {\it Transactions of the American Mathematical Society},  114(2):309--319, 1965.


\bibitem{st_billingsley}
Billingsley, P., {\it Ergodic Theory and Information}, Robert E. Krieger Publishing Co., Huntington, New York, 1978.


\bibitem{Bunimovich}
Bunimovich, L.A., Dani, S.G., Dobrushin, R.L., Jakobson, M.V., Kornfeld, I.P., Maslova, N.B., Pesin, Ya. B., Sinai, Ya. G., Smillie, J., Sukhov, Yu. M.,
Vershik, A.M., {\it Dynamical Systems, Ergodic Theory, and Applications}, $2^{nd}$ ed., Encyclopaedia of Mathematical Sciences, Vol. 100, Springer-Verlag,
Berlin, 2000.


\bibitem{st_halmos}
Halmos, P.R., {\it Lectures on Ergodic Theory}, The Mathematical Society of Japan, Tokyo, 1956.


\bibitem{erg_Kleinbock}
Kleinbock, D., Shah, N., Starkov, A., ``Dynamics of subgroup actions on homogeneous spaces of Lie groups and applications to number theory,'' in {\it Handbook of Dynamical Systems, Vol. 1A} (B. Hasselblatt, A. Katok, eds.), Chapter 11 (pp. 813--930) Elsevier, 2002.


\bibitem{st_manebook}
Ma\~{n}\'{e}, R., {\it Ergodic Theory and Differentiable Dynamics},  (translated from the Portuguese by Silvio Levy), Springer-Verlag, Berlin ; New York, 1987.

\bibitem{erg_Margulis}
Margulis, G.A., Nevo, A., Stein, E.M.,
``Analogs of Wiener's Ergodic Theorems for Semisimple Groups II,''
{\it Duke Math, J.} 103(2):233--259, 2000.

\bibitem{st_Mooreerg}
Moore, C.C., ``Ergodicity of flows on homogeneous spaces,'' {\it Amer. J. Math.} 88: 154--178, 1966.

\bibitem{st_Moser}
Moser, J., Phillips, E., Varadhan, S., {\it Ergodic Theory (A Seminar)}, Courant Institute, NYU, New York, 1975.

\bibitem{st_parry}
Parry, W., {\it Topics in Ergodic Theory}, Cambridge University Press, Cambridge, England, 1981.

\bibitem{st_petersen}
Petersen, K., {\it Ergodic Theory}, Cambridge University Press, Cambridge, England, 1983.

\bibitem{st_rokhlin}
Rokhlin, V.A., ``Lectures on the entropy theory of transformations with invariant measure,''
{\it Usp. Mat. Nauk.} 22: 3--56, 1967; {\it Russian Math. Surveys} 22:1--52, 1967.

\bibitem{erg_Ruelle}
Ruelle, D., ``Ergodic Theory of Differentiable Dynamical Systems,'' {\it Publ. IHES} 50:275--306, 1979.

\bibitem{st_sinaient}
Sinai, Ya. G., ``On the Notion of Entropy of Dynamical Systems,''
{\it Dokl. Acad. Sci. USSR}  124(4(: 768--771, 1959.



\bibitem{st_Templeman}
Templeman, A., {\it Ergodic Theorems for Group Actions: Informational and Thermodynamical Aspects},
Kluwer Academic Publishers, Dordrecht, The Netherlands, 1992.

\bibitem{erg_Ulam}
Ulam, S.M., von Neumann, J., ``Random Ergodic Theorems,'' {\it Bull. Amer. Math. Soc.}, 51(9):660--, 1947



\bibitem{st_walters}
Walters, P., {\it An Introduction to Ergodic Theory}, Springer-Verlag, New York, 1982.





\bibitem{pk_boothroyd}
Boothroyd, G., Redford, A.H.,  {\it Mechanized Assembly: Fundamentals of
parts feeding, orientation, and mechanized assembly}, McGraw-Hill,
London, 1968.

\bibitem{pk_booth1}
Boothroyd G., {\it Assembly Automation and Product Design}, $2^{nd}$ ed., CRC Press, Boca Raton, FL, 2005

\bibitem{pk_Homem}
de Mello, L.S.H., Lee, S., eds., {\it Computer-Aided Mechanical Assembly Planning},
Kluwer, Boston, 1991.

\bibitem{pk_erdmann}
Erdmann, M.A., Mason, M.T., ``An Exploration of Sensorless
Manipulation", I.E.E.E. {\it Journal of Robotics and Automation},
Vol. 4, No. 4, pp. 369--379, August 1988.

\bibitem{pk_liu}
Liu, Y., Popplestone, R.J., ``Symmetry Groups in Analysis of
Assembly Kinematics" , \emph{ICRA 1991}, pp. 572 -- 577, Sacramento,
CA, April 1991.

\bibitem{pk_Whitney}
Whitney, D.E., {\it Mechanical Assemblies}, Oxford University Press, New York, 2004.


\bibitem{pk_blaschke2}
Blaschke, W., {\it Vorlesungen \"{u}ber Integralgeometrie},
Berlin, Deutscher Verlag der Wissenschaften, 1955.


\bibitem{pk_chern}
Chern, S.-S.,
``On the Kinematic Formula in the Euclidean Space of $N$ Dimensions,'' {\it
American Journal of Mathematics}, Vol. 74, No. 1 (Jan., 1952), pp. 227--236





%
%

%


\bibitem{pk_Klain}
Klain, D.A., Rota, G.-C., {\it Introduction to Geometric Probability}, Cambridge University Press, 1997.



\bibitem{pk_poin}
Poincar\'{e}, H., {\it Calcul de Probabilit\'{e}s}, $2^{nd}$ ed., Paris 1912.
(reprinted by BiblioLife in 2009)

\bibitem{pk_santalo}
Santal\'{o}, L., {\it Integral Geometry and Geometric Probability},
Cambridge University Press, 2004 (originally published in 1976 by Addison-Wesley)



\bibitem{SchneiderWeilbook}
Schneider, R., Weil, W.,
{\it Stochastic and Integral Geometry}, Springer-Verlag, Berlin, 2008.


%
%
%
%

\bibitem{pk_Zhang}
Zhang, G., ``A sufficient condition for one convex body containing another,'' {\it Chinese Ann. of
Math.} 9B(4) (1988), 447--451.

\bibitem{pk_zhouj92}
Zhou, J., ``A Kinematic Formula and Analogues of Hadwiger's Theorem in Space,''
{\it Contemporary Mathematics} 140, pp. 159-167, American Mathematical Society, 1992.

\bibitem{pk_zhouj95}
Zhou, J., ``When Can One Domain Enclose Another in $\IR^3$?,''
{\it J. Austral. Math. Soc. (Series A)} 59:266--272, 1995

\bibitem{pk_zhouj98}
Zhou, J., ``Sufficient Conditions for One Domain to
Contain Another in a Space of Constant Curvature,''
{\it Proc. AMS}, 126(9):2797--2803, 1998.

\bibitem{pk_Karnik}
Karnik, M., Gupta, S.K., Magrab, E.B.,
``Geometric algorithms for containment analysis of rotational parts,'' {\it Computer Aided Design},
37(2):213-230, 2005.

\bibitem{kin-contain}
Ruan, S., Ding, J., Ma, Q. and Chirikjian, G.S., 2019. The kinematics of containment for N-dimensional ellipsoids. Journal of Mechanisms and Robotics, 11(4).

\end{thebibliography}
\end{document}